\documentclass[11pt]{amsart}

\usepackage{amsfonts,amsmath,amssymb,amsthm,mathrsfs,url,color,latexsym,rawfonts,lscape,graphicx}
\usepackage{hyperref}
\usepackage{enumerate} 

\theoremstyle{plain}
  \newtheorem{theorem}{Theorem}[section]
  \newtheorem{lemma}{Lemma}[section]

  \newtheorem{corollary}{Corollary}[section]

   \newcommand{\beqn}{\begin{eqnarray}}
   \newcommand{\eeqn}{\end{eqnarray}}
   \newcommand{\beqs}{\begin{eqnarray*}}
   \newcommand{\eeqs}{\end{eqnarray*}}
   \newcommand{\ban}{\begin{eqnarray*}}
   \newcommand{\nan}{\end{eqnarray*}}
   \newcommand{\beq}{\begin{equation}}
   \newcommand{\eeq}{\end{equation}}

  \newcommand{\RR}{{\mathbb R}}

  \newcommand{\R}{\RR}

\newcommand{\p}{\partial}

\newcommand{\Om}{\Omega}
\newcommand{\pom}{{\p\Om}}
\newcommand{\bom}{{\overline\Om}}

\renewcommand{\det}{\mbox{det}}
  
  \newcommand{\dist}{\mbox{dist}}

\def\rhoprime{\rho}

  \numberwithin{equation}{section}
  \numberwithin{figure}{section}

\begin{document}

\title[Regularity for the optimal transportation]
{Global regularity of optimal mappings \\ in non-convex domains}

\date\today

\author[S. Chen]
{Shibing Chen}
\address
{School of Mathematical Sciences,
University of Science and Technology of China,
Hefei, 230026, P.R. China}
\email{chenshib@ustc.edu.cn}

\author[J. Liu]
{Jiakun Liu}
\address
	{School of Mathematics and Applied Statistics,
	University of Wollongong,
	Wollongong, NSW 2522, AUSTRALIA}
\email{jiakunl@uow.edu.au}

\author[X.-J. Wang]
{Xu-Jia Wang}
\address
{Centre for Mathematics and Its Applications,
The Australian National University,
Canberra, ACT 0200, AUSTRALIA}
\email{Xu-Jia.Wang@anu.edu.au}

\thanks{This work was supported by ARC FL130100118 and ARC DP170100929.}



\begin{abstract}
In this paper, 
we establish a global regularity result for the optimal transport problem with the quadratic cost,
where the domains may not be convex. 
This result is obtained by a perturbation argument,
using a recent global regularity of optimal transportation in convex domains by the authors.
\end{abstract}

\maketitle

\baselineskip=16.2pt
\parskip=3pt

\section{Introduction}

The regularity of optimal mappings is a core issue in optimal transport problem \cite{Ca96,V2}, which can be described as follows:
Suppose there is a source domain $\Om\subset \mathbb{R}^n$ with density $f$ and a target domain $\Omega^*\subset \mathbb{R}^n$ with density $g$ satisfying the balance condition
\beq\label{BC}
\int_\Om f=\int_{\Om^*}g.
\eeq
Given a cost function $c(x,y): \Om\times \Om^*\rightarrow \mathbb{R}$,
one asks for the existence and regularity of an optimal mapping $T$ that minimises the transport cost 
\begin{equation}\label{ot}
\mathcal C(T) = \int_{\Om} c(x, T(x)) f(x)dx
\end{equation}
among all measure preserving maps.
A mapping $T : \Om\to\Om^*$ is called measure preserving, denoted as $T_\sharp f=g$, if for any Borel set $E\subset\Om^*$
$$\int_{T^{-1}(E)} f=\int_E g.$$

The optimal transport problem was first introduced by Monge \cite {Mo} with the natural cost function $c(x, y)=|x-y|$, and was extensively studied since after.
When the cost function 
\beq\label{xdoty}
c(x, y)=x\cdot y,
\eeq
or equivalently the quadratic cost $c(x, y)=\frac12|x-y|^2$,
the existence and uniqueness of the optimal mapping were obtained by Brenier \cite{Bre91}.
It was shown that the optimal mapping $T=Du$ is the gradient of a convex potential function $u$, which satisfies
\begin{equation}\label{ma1}
	\det\, D^2u(x) = \frac{f(x)}{g(Du(x))}, 
\end{equation}
\begin{equation}\label{bdry}
		Du(\Omega) = \Omega^*.
\end{equation}

In this paper we study the regularity of solutions to the above boundary value problem.
The densities $f, g$ are always assumed to satisfy \eqref{BC} and
\beq\label{fg}
c_0\le f, g\le c_1
\eeq
for two positive constants $c_1\ge c_0>0$, which makes the equation \eqref{ma1} elliptic. 

Due to its applications in optimal transportation and in many other areas, 
the boundary value problem \eqref{ma1}--\eqref{bdry} has received huge attention
and been studied intensively in recent years \cite {Eva, V2}.
Assuming both domains $\Om, \Om^*$ are convex Pogorelov \cite{P64} obtained a generalised solution in the sense of Aleksandrov. 
In \cite{Bre91}, Brenier showed the existence and uniqueness of solutions in another weak sense, which is equivalent to Aleksandrov's solution when $f, g$ satisfy \eqref{fg} and the target domain is convex. But we also refer the reader to \cite{Bak} for extension of Aleksandrov's generalised solutions. 
The interior regularity was developed by many people, see for example \cite{C1,Ca92,DFS,JW,TW08}, the books \cite{F,GT,Gut} and references therein. 
Very recently, a new proof was found in \cite {WW}, using the Green function of the linearised Monge-Amp\`ere equation.

For the global regularity, assuming both domains $\Om, \Om^*$ are uniformly convex and $C^{3,1}$ smooth, and the densities $f\in C^{1,1}(\bom)$, $g\in C^{1,1}(\overline{\Omega^*})$, the global smooth solution was first obtained by Delano\"e \cite{D91} for dimension two and later extended to high dimensions by Urbas \cite{U1}.
In a milestone work \cite{C96}, Caffarelli proved that $u\in C^{2,\alpha'}(\bom)$ for some $\alpha'\in (0, \alpha)$, if $\Om, \Om^*$ are uniformly convex with $C^2$ boundary, and the densities $f, g\in C^\alpha$. 
The uniform convexity of domains plays a critical role in the above mentioned papers \cite{C96, D91, U1}, which is also necessary for the global regularity of solutions to other boundary value problems such as in \cite{LTU, Sa, TW08a}.
In a recent paper \cite{CLW}, the authors removed this condition for the problem  \eqref{ma1}--\eqref{bdry} and obtained the following

\begin{theorem}[\cite{CLW}]\label{main}
Assume that $\Omega$ and $\Omega^*$ are bounded convex domains in $\mathbb{R}^n$ 
with $C^{1,1}$ boundaries.
Let $u$ be a convex solution to  \eqref{ma1}--\eqref{bdry}. We have the following estimates:
\begin{itemize}
\item[(i)] If $f\in C^\alpha(\bom)$, $g\in  C^\alpha(\overline{\Om^*})$, for some $\alpha\in(0,1)$, then
\beq\label{esti}
\|u\|_{C^{2,\alpha}(\bom)}\le C,
\eeq
where $C$ is a constant depending on $n, \alpha, f, g, \Om$, and $\Om^*$.
\item[(ii)]  If $f\in C^0(\bom)$, $g\in  C^0(\overline{\Om^*})$, then 
\beq\label{esti1}
\|u\|_{C^{1,\beta}(\bom)}\le C_\beta \quad \forall\, \beta\in(0,1),\quad\mbox{ and }\ \ \|u\|_{W^{2,p}(\bom)}\le C_p \quad \forall\, p\geq1,
\eeq
where the constants $C_\beta, C_p$ depend on $n, f, g, \Om, \Om^*$, and on $\beta, p$, respectively.
\end{itemize}
\end{theorem}

In this paper, we relax furthermore the convexity condition of the domains.
For any given non-convex domain $\Omega^*$, it was shown that there exist smooth and positive densities $f, g$ such that the potential function $u$ is not $C^1$ \cite{MTW}.
However, for fixed positive and smooth densities $f, g$, by Theorem \ref{main} and a perturbation argument, in this paper we can show that $u$ is smooth up to the boundary when the domains are small perturbations of convex domains, but may not be convex themselves.

Given a bounded domain $\Lambda\subset\R^n$, we say $\Lambda$ is \emph{$\delta$-close} to $\Om$ in $C^{1,1}$ norm, if there exists a bijective mapping $\Phi : \Om\to\Lambda$ such that $\Phi\in C^{1,1}(\bom)$ and
\beq\label{deldis}
\|\Phi-I\|_{C^{1,1}(\bom)} \leq \delta
\eeq 
where $I:\Om\to\Om$ is the identity mapping. 
A localised definition of $\delta$-closeness is given in \S2.1.
Now we can state our main theorem.

\begin{theorem}\label{ptheorem}
Let $\Lambda$  and $\Lambda^*$ be $C^{1,1}$ domains that are
$\delta$-close to $\Omega$  and $\Omega^*$ in $C^{1,1}$ norm, respectively,
where $\Om$ and $\Om^*$ are bounded convex domains with $C^{1,1}$ boundaries. 
Suppose that $f, g$ satisfy \eqref{BC} and \eqref{fg}, and $f\in C^{\alpha}(\overline{\Lambda}),g\in C^{\alpha}(\overline{\Lambda^*})$,
for some $\alpha\in (0,1)$.
Then there exists a small constant $\delta_0>0$ depending only
on $\Omega, \Omega^*, \alpha, c_0, c_1, \|f\|_{C^{\alpha}(\Lambda)}$ and $\|g\|_{C^{\alpha}(\Lambda^*)}$,
such that the potential function $u\in C^{2,\alpha}(\overline{\Lambda}) $,  provided $\delta<\delta_0.$
\end{theorem}

The proof of Theorem \ref{ptheorem} is carried out in the following two sections: 
In \S2, we prove that $u\in C^{1,\beta}(\overline\Lambda)$ for any given $\beta\in (0,1)$, by using a localisation and iteration argument.
Then in \S3, by adapting a perturbation argument from \cite[\S5]{CLW}, we obtain $u\in C^{2,\alpha}(\overline{\Lambda})$.
As a byproduct we also obtain the following global $W^{2,p}$ estimate. 

\begin{theorem}\label{2ptheorem}
Let $\Lambda, \Lambda^*$ be as in Theorem \ref{ptheorem}.
Suppose $f, g$ satisfy \eqref{BC} and \eqref{fg}, and  $f\in C(\overline{\Lambda}),g\in C(\overline{\Lambda^*})$.
Then $\forall\ p\ge 1$, $\exists$ a small constant $\delta_0>0$ depending only
on $\Omega, \Omega^*, p, c_0, c_1, f, g $,
such that the potential function $u\in W^{2,p}(\overline\Lambda) $,  provided $\delta<\delta_0.$
\end{theorem}

In the last section \S4, we give some interesting applications of Theorems \ref{ptheorem}, \ref{2ptheorem} in the free boundary problems, minimal Lagrangian diffeomorphism, and optimal transportation with general costs, and hope to motivate future study in these areas.

\vspace{10pt}

\section{$C^{1,\beta}$ regularity}\label{s2}

In this section, by using a perturbation and an iteration argument, we prove the $C^{1,\beta}$ estimate, for any given $\beta\in (0,1)$.
Some of our arguments are inspired by those in \cite{CF1,CGN,DF}. 
In \cite{DF}, De Philippis and Figalli \cite{DF} obtained a partial regularity result for optimal transport problem with general cost functions.
In \cite{CF1}, Figalli and the first author obtained a global regularity under a small perturbation of the quadratic cost. 
In \cite{CGN}, the regularity of optimal transport is obtained for the cost $|x-y|^p$ when $p$ is close to $2$.

\begin{lemma}\label{alpha1}
Under the hypotheses of Theorem \ref{ptheorem},
for any given  $\beta\in (0, 1)$,
there exits a small constant $\delta_0>0$ 
depending only on $\Omega, \Omega^*, \alpha, \beta, c_0, c_1, \|f\|_{C^{\alpha}(\Lambda)}$ and $\|g\|_{C^{\alpha}(\Lambda^*)}$,
such that the potential function $u\in C^{1,\beta}(\overline{\Lambda})$, provided $\delta<\delta_0.$
\end{lemma}

By the interior regularity of the Monge-Amp\`ere equation \cite{C1,JW},  
it suffices to prove Lemma \ref{alpha1} near the boundary. 
The proof is divided into four subsections following the strategy that:
First in \S\ref{ss21}, we show that since the domains are small perturbations of convex domains, the potential $u$ is also a small perturbation of a $C^{2,\alpha}$ potential function $\tilde u$ for convex domains. 
Then in \S\ref{ss22}, we localise the problem by rescaling it near a boundary point, and show that $u$ is close to the parabola $\frac{1}{2}|x|^2$ (given by the second order Taylor expansion of $\tilde u$ at the origin), the densities are close to constants, and the boundaries of domains are close to be flat.
Next in \S\ref{ss23}, we prove that $u$ is close to a convex function $w$ solving an optimal transport problem with constant densities. In addition, $w$ is smooth, and thus $u$ is even closer to a parabola (given by the second order Taylor expansion of $w$ at the origin) inside a small sub-level set $S_{h_0}[u]$. 
Last in \S\ref{ss24}, by rescaling $S_{h_0}[u]$ at scale $1$ and iterating the above steps, we obtain that $u$ is $C^{1,\beta}$ at the origin for any given $\beta\in (0,1)$, and thus prove Lemma \ref{alpha1}.

\subsection{Comparison with a solution over convex domains}\label{ss21}

Let $\Omega$ be a $C^{1,1}$ convex domains in $\R^n$. 
For any given point $x_0=0\in \partial\Omega $, 
there exists  a small ball $B_{r}=B_r(x_0)$, of which the radius $r$ is independent of $x_0$, 
such that after a rotation of the coordinates, locally the boundary can be expressed as
$$\partial\Omega\cap B_r=\{x=(x', x_n)\in\R^n\ : \ x_n=\eta(x')\} , $$ 
where $x'=(x_1, \cdots, x_{n-1})$, and $\eta$ is a $C^{1,1}$ convex function satisfying 
$$\eta\geq0,\ \ \ \eta(0)=0\ \ \ \ \text{and}\ \  D\eta(0)=0 . $$

Let $\Lambda\subset\R^n$ be a $C^{1,1}$ domain and is $\delta$-close to $\Omega$ in $C^{1,1}$ norm. 
From the definition \eqref{deldis}, 
locally $B_r\cap \partial \Lambda$ can be represented as the graph of a $C^{1,1}$ function $\rho$ such that 
\beq\label{dd2}
\|\rho-\eta\|_{C^{1,1}(B'_r)}\leq \delta,
\eeq
for all $x_0\in \partial\Omega$, where $B'_r$ is a ball in $\mathbb{R}^{n-1}$ with the radius $r>0$ independent of $x_0$. 
In fact, the bijection $\Phi$ in \eqref{deldis} can be defined such that for $x\in\Om$ close to $\pom$, $\Phi(x)=(x', x_n+(\rho-\eta)(x'))$.
Then \eqref{deldis} is equivalent to \eqref{dd2} due to a finite covering and the compactness of $\pom$. 

\vskip5pt

Recall that $f, g$ are the given densities supported on $\Lambda, \Lambda^*$, respectively, 
and $u$ is the potential function of the optimal transport from $\Lambda$ to $\Lambda^*.$
Under the hypotheses of Theorem \ref{ptheorem},
let $f_1, g_1\in C^{\alpha}(\mathbb{R}^n)$ be the extensions of $f, g$ with the same H\"older exponent 
and satisfy \eqref{fg} for some positive constants $c_0, c_1$ (that may be different to the constants in \eqref{fg}).
Let $\tilde{f}, \tilde{g}$ be the restriction of $f_1, g_1$ on $\Omega, \Omega^*$.
Let $\tilde{u}$ be the potential function of the optimal map from $(\Omega, \tilde{f})$ to $(\Omega^*, \lambda\tilde{g}),$
where the constant $\lambda$ is chosen such that $\int_\Omega \tilde{f}=\int_{\Omega^*}\lambda\tilde{g}.$
Apparently $\lambda\rightarrow 1$ as $\delta\rightarrow 0$.  
Without loss of generality we may assume directly that $\lambda=1$.
Replacing $\Om$ by $\hat x+(1+C\delta)(\Om-\hat x)$ for some interior point $\hat x\in\Om$, (similarly to $\Om^*$),
we may also assume that $\Lambda\subset \Omega$ and $\Lambda^*\subset \Omega^*$.   

\vskip5pt

Therefore, in the following we always have
$\Lambda\subset \Omega$, $\Lambda^*\subset \Omega^*$, and 
 $\Lambda, \Lambda^*$ are $\delta$-close to the convex domains $\Om, \Om^*$ respectively, and that 
\beq\label{fgt}
\int_\Omega \tilde{f}=\int_{\Omega^*} \tilde{g} . 
\eeq
Let $\tilde{u}$ be the potential function of the optimal transport from $(\Omega, \tilde{f})$ to $(\Omega^*, \tilde{g})$.  
For simplicity, we introduce the notation 
$$\|u-\tilde u\|_{\infty}=\sup_{x\in\Lambda} \big\{ [u(x)-u(x_0)] - [\tilde u(x)-\tilde u(x_0)] \big\} , $$
where $x_0$ is the mass centre of $\Lambda$.
{By adding a suitable constant to $u$ such that $u(x_0)=\tilde u(x_0)$, one can see that $\|u-\tilde u\|_{\infty}=\|u-\tilde{u}\|_{L^\infty(\Lambda)}$.}
  
\begin{lemma}\label{vclose}
There exists a positive function $\omega: \mathbb{R}_+\rightarrow\mathbb{R}_+$,
depending only on $c_0, c_1$, the inner and outer radii of $\Omega,   \Omega^* $, 
with the property $\omega(\delta)\rightarrow 0$ as $\delta\rightarrow 0$, 
such that 
\beq\label{Esti0}
\|u-\tilde{u}\|_{\infty} \leq \omega(\delta).
\eeq
\end{lemma}

\begin{proof}
Suppose to the contrary that there exist
$\{\Lambda_k, \Lambda_k^*, f_k, g_k\}$, 
and the convex approximations $\{\Omega_k, \Omega_k^*, \tilde{f}_k, \tilde{g}_k\}$ satisfying the above conditions
such that the associated potential functions satisfy 
\beq\label{appk}
\|u_k-\tilde{u}_k\|_\infty >\delta_0
\eeq
for a fixed small constant $\delta_0$ independent of $k$.
Passing to a subsequence and taking the limit we have
\begin{itemize}
\item $\Lambda_k, \Lambda_k^*$ (and also $\Omega_k, \Omega^*_k$) converge to convex domains
$\Omega_\infty$, $\Omega^*_\infty$ in Hausdorff distance respectively, as $k\rightarrow \infty $.  

\item  $f_k\chi_{\Lambda_k}, g_k\chi_{\Lambda^*_k}$ 
(and also $\tilde{f}_k\chi_{\Omega_k}, \tilde{g}_k\chi_{\Omega^*_k}$) 
converge weakly  to $f_\infty\chi_{\Omega_\infty}$, $g_\infty\chi_{\Omega^*_\infty}$,
where $f_\infty, g_\infty$ satisfy \eqref{fg} in $\Omega_\infty$, $\Omega^*_\infty$.

\item  $u_k\rightarrow u_\infty, \tilde{u}_k\rightarrow \tilde{u}_\infty$ for some convex functions $u_\infty, \tilde{u}_\infty$.  
\end{itemize}
Since \eqref{appk} is independent of $k$, we obtain 
\beq\label{appk2}
\|u_\infty-\tilde{u}_\infty\|_\infty \geq\delta_0.
\eeq

On the other hand, $u_k$ are potential functions of the optimal transport from $(\Lambda_k, f_k)$ to $(\Lambda_k^*, g_k)$,
$u_\infty$ is the potential function of the optimal transport from $(\Om_\infty, f_\infty)$ to $(\Om^*_\infty, g_\infty)$.
Hence  $Du_\infty$ and $D\tilde{u}_\infty$ are both optimal maps from $(\Omega_\infty, f_\infty)$ to $(\Omega^*_\infty, g_\infty)$.
It follows that $Du_\infty =D\tilde{u}_\infty$ a.e. and thus 
$u_\infty =\tilde{u}_\infty+c$ for some constant $c$, which implies that $\|u_\infty-\tilde{u}_\infty\|_\infty = 0$ contradicting with \eqref{appk2}.
The lemma is proved.
\end{proof}

\vskip5pt

\subsection{Localisation near a boundary point}\label{ss22}

Let $0\in \partial\Lambda$ be a boundary point, and locally the boundary is given by a $C^{1,1}$ function $\rho$ such that
$$\partial\Lambda=\{x : x_n=\rhoprime (x')\}\ \ \ \text{with}\ \  \rhoprime (0)=0, \ D\rhoprime (0)=0.$$
Recall that $\tilde{u}$ is the potential function of the optimal transport from $C^{1,1}$ convex domains $\Omega$ to $\Omega^*$. 
From Theorem \ref{main} (i), $\tilde{u}\in C^{2,\alpha}(\overline{\Omega}) $.
Hence by subtracting a linear function and performing
an affine transformation, we may assume that
\begin{equation}\label{taylor1}
\tilde{u}=\frac{1}{2}|x|^2+O(|x|^{2+\alpha}),
\end{equation}
near the origin. 
By adding a constant, we also assume that $u(0)=\tilde u(0)=0$. 
Then from Lemma \ref{vclose}, near the origin, one has
\beqs
\left|u-\frac12|x|^2\right| \leq \omega(\delta)+O(|x|^{2+\alpha}),
\eeqs
which implies that $\p u(0)$ converges to $0$ as $\delta\rightarrow 0$, where $\p u$ is the sub-differential of $u$.   
Hence, up to an affine transformation (converging to identity as $\delta\rightarrow 0$) there is a $C^{1,1}$ function $\rho^*$ such that locally
$$  \partial\Lambda^*= \{x: x_n=\rhoprime ^*(x')\}\ \ \ \text{with}\ \  
                 \rhoprime ^*(0)\to 0\ \text{and} \ D\rhoprime ^*(0)\to 0,\ \ \text{ as}\ \delta\to 0. $$

\vskip5pt

Denote $U=B_{\epsilon_0}\cap \Lambda$ and $U^*=\partial u(U)\cap \Lambda^* $,  
where $\epsilon_0>0$ is a small constant to be determined later.
By subtracting a linear function $\ell$ from $u$ with $|D\ell|\rightarrow 0$ as $\delta\rightarrow 0$, 
we may assume that $u(0)=0$ and $0\in \partial u(0)$.
By Lemma \ref{vclose} and \eqref{taylor1}, we have
\beq\label{cd0}
\|u-\frac12|x|^2\|_{L^\infty(U)} \leq \omega(\delta) + C\epsilon_0^{2+\alpha},
\eeq
and by the convexity of $u$,
\beq\label{imap}
\{x: x_n>\rhoprime ^*(x')\}\cap B_{\epsilon_0/2}\subset U^*\subset \{x: x_n>\rhoprime ^*(x')\}\cap B_{2\epsilon_0}.
\eeq

Make a rescaling 
\beq\label{resc}
x\mapsto x/\epsilon_0
\eeq
such that $B_{\epsilon_0}$ becomes $B_1$. Denote
\beqs
u'(x):=\frac{1}{\epsilon_0^2}u(\epsilon_0 x),\quad \mbox{ in the domain }\ U':=\frac{1}{\epsilon_0}U.
\eeqs

\begin{lemma}\label{rele}
For any given $\eta_0>0$ small, by choosing $\delta, \epsilon_0$ sufficiently small, one has
\begin{equation} \label{cd1}
\|u'-\frac{1}{2}|x|^2\|_{L^\infty(U')}<\eta_0.
\end{equation}
Moreover, under the rescaling \eqref{resc}, the densities $f, g$ tend to be constant and boundaries $\p\Lambda, \p\Lambda^*$ tend to be flat, as $\epsilon_0\to0$. 
\end{lemma}

\begin{proof}
From \eqref{cd0}, \eqref{resc} and the fact that $U\subset B_{\epsilon_0}$, it is straightforward to see that for any given $\eta_0>0$ small, 
\begin{equation*} 
\|u'-\frac{1}{2}|x|^2\|_{L^\infty(U')}\leq \frac{\omega(\delta)}{\epsilon_0^2}+C\epsilon_0^\alpha<\eta_0
\end{equation*}
provided $\delta, \epsilon_0$ are sufficiently small.

By the rescaling \eqref{resc}, the domains become 
$U'=\frac{1}{\epsilon_0}U$, ${U^*}'=\frac{1}{\epsilon_0}U^*$,
which are locally given respectively by 
$${\begin{aligned}
     x_n> &\ {\tilde\rho}(x'):=\frac{1}{\epsilon_0}\rhoprime (\epsilon_0x'),\\ 
   x_n> &\ {\tilde\rho}^*(x'):=\frac{1}{\epsilon_0}\rhoprime ^*(\epsilon_0x').
\end{aligned}} $$
From \eqref{imap} it is easy to see that    
\begin{equation}\label{incl11}
B_{1/2}\cap  \{x_n> {\tilde\rho}(x')\}\subset  U'\subset B_{2}\cap  \{x_n> {\tilde\rho}(x')\},
\end{equation}
\begin{equation}\label{incl12}
B_{1/2}\cap  \{x_n> {\tilde\rho}^*(x')\}\subset  {U^*}'\subset B_{2}\cap  \{x_n> {\tilde\rho}^*(x')\}.
\end{equation}
For any given small $\eta_0>0$, by direct computation one can check that 
\begin{equation}\label{cd2}
{\begin{aligned}
 & \|{\tilde\rho}\|_{C^{1,1}(B'_{1/2})}\leq \epsilon_0\|{\rho}\|_{C^{1,1}(B'_{\epsilon_0})}<\eta_0, \ \\
 & \|{\tilde\rho^*}\|_{C^{1,1}(B'_{1/2})}\leq \epsilon_0\|{\rho^*}\|_{C^{1,1}(B'_{\epsilon_0})}<\eta_0, 
 \end{aligned}}
\end{equation}
provided $\epsilon_0$ are sufficiently small, where $B'_r$ indicates a ball in $\mathbb{R}^{n-1}$ with radius $r$. 
Therefore, under the rescaling \eqref{resc} the boundaries $\p\Lambda$ and $\p\Lambda^*$ tend to be flat as $\epsilon_0\to0$. 

Correspondingly, the density functions become $f_1(x)=f(\epsilon_0 x)$ in $U'$, and $g_1(x)=g(\epsilon_0 x)$ in ${U^*}'$. 
Similarly as \eqref{cd2} one has for any small $\eta_0>0$,
\begin{equation}\label{cd3}
{\begin{aligned}
 & \|f_1-f(0)\|_{C^\alpha(U')}<\eta_0,\ \\  
 & \|g_1-g(0)\|_{C^\alpha({U^*}')}<\eta_0,
 \end{aligned}}
 \end{equation}
provided $\epsilon_0$ are sufficiently small.      
Note that without loss of generality, we can always assume that $f(0)=g(0)=1.$ 
\end{proof}

\vskip 5pt

\subsection{Approximation by a smooth solution}\label{ss23}

By the rescaling in Lemma \ref{rele}, having  \eqref{incl11}--\eqref{cd3} we show that $u'$ can be approximated by a smooth convex function $w$ solving an optimal transport problem with constant densities. 
And then we deduce that $u$ is even closer to another parabola (comparing with \eqref{cd1}) in a small sub-level set 
\beq\label{sections}
S_{h_0}[u']:=\{x\in U' : u'(x)<h_0\},\quad\mbox{ for $h_0>0$ small.}
\eeq

Let
$$\delta_1:=\sup_{x'\in B'_2} |{\tilde\rho}(x')|+\sup_{x'\in B'_2} |{\tilde\rho}^*(x')|  $$ 
and denote $U'_-$ (resp. ${U^*}'_-$)  the reflection of $U'$ (resp. ${U^*}'$) with respect to the hyperplane $\{x_n=-\delta_1\} $.  
Let $z=(0,\cdots, 0, -\delta_1)$,
$${\begin{aligned}
  & \mathcal{D}_1:=U'\cup U'_-\cup B_{\frac{1}{2}}(z),\\ 
 & \mathcal{D}_2:=\lambda\left({U^*}'\cup {U^*}'_-\cup B_{\frac{1}{2}}(z)\right), 
 \end{aligned}} $$  
where the constant $\lambda$ is chosen such that
$|\mathcal{D}_1|=|\mathcal{D}_2| $.  
Note that $\delta_1\rightarrow 0$ as $\delta\rightarrow 0$,   
both $\mathcal{D}_1$ and $\mathcal{D}_2$ are symmetric with respect  to $\{x_n=-\delta_1\}$, 
and 
\begin{equation}\label{conta1}
B_{1/3}\subset \mathcal{D}_1, \mathcal{D}_2 \subset B_3.
\end{equation}
 Note also that $\lambda\rightarrow 1$ as $\delta\rightarrow 0$. 

Let $w$ be the convex function solving 
$(\partial w)_\sharp \chi_{\mathcal{D}_1}=\chi_{\mathcal{D}_2}$ with $w(0)=u'(0) $. 
Namely, $w$ is a solution to
\beq\label{mae01}
\bigg\{{\begin{aligned}
	 &\det\, D^2w(x) = 1 \ \ \text{in}\ \mathcal D_1,\\
	 &Dw(\mathcal D_1) =\mathcal D_2.
	 \end{aligned}} \eeq
By the symmetry of the data and the uniqueness of optimal transport maps, 
we see that $w' $,  the restriction of $w$ on to $\mathcal{D}_1\cap \{x_n>-\delta_1\} $,  solves 
$(\partial w')_\sharp \chi_{\mathcal{D}_1\cap \{x_n>-\delta_1\}}=\chi_{\mathcal{D}_2\cap \{x_n>-\delta_1\}} $.  
Namely,  $\p w' (\mathcal{D}_1\cap \{x_n>-\delta_1\})=\mathcal{D}_2\cap \{x_n>-\delta_1\} $.  
By a compactness argument similar to that of Lemma \ref{vclose}, we have  
\begin{equation}\label{cpt11}
\|w'-u'\|_{L^\infty(B_{{1}/{3}}\cap \{x_n>{\tilde\rho}(x')\})} 
    \leq \omega(\delta)\rightarrow 0,\quad \ \text{as}\ \  \delta\rightarrow 0.
\end{equation}

By \eqref{cd1} we then have 
 \begin{equation}\label{ap1}
 \|w'-\frac{1}{2}|x|^2\|_{L^\infty(B_{{1}/{3}}\cap \{x_n>{\tilde\rho}(x')\})} 
       \leq \omega(\delta)+\eta_0\rightarrow 0,\quad \ \text{as}\ \  \delta, \eta_0\rightarrow 0.
 \end{equation}
 By the symmetry of $w $,  we also have that  
  \begin{equation}\label{ap2}
 \|w-\frac{1}{2}|x|^2\|_{L^\infty(B_{{1}/{4}}\cap \{x_n>{\tilde\rho}(x')\}  )}
      \leq \omega(\delta)+\eta_0\rightarrow 0,\quad \ \text{as}\ \  \delta, \eta_0\rightarrow 0.
 \end{equation}
Noting that by \eqref{ap2}  and symmetry of $w$, we have $\partial w(B_{1/5})\subset B_{1/4}$, provided $\delta, \eta_0$ are sufficiently small.
Since  $w$ is a solution to \eqref{mae01}, by the interior estimates \cite{GT} we obtain 
\beq\label{intC3}
\|w\|_{C^3(B_{1/5})}\leq C
\eeq
for a universal constant $C>0$.

\vskip5pt

Since $\|w\|_{C^3(B_{1/5})}\leq C$ and $w(0)=0$,  we have the Taylor expansion
\beq\label{Taylorw}
w(x)=Dw(0)\cdot x+\frac{1}{2}D^2w(0)x\cdot x+O(|x|^3).
\eeq
From \eqref{cpt11} and by a compactness argument, we claim that 
$|Dw(0)|\rightarrow 0$ as $\delta \rightarrow 0$.
Indeed, in the limit profile as $\delta \rightarrow 0$ we have $w=u'$ on
$B_{\frac{1}{3}}\cap \{x_n>{\tilde\rho}(x')\}$ which implies that $Dw(0)=Du'(0)=0$. 
Therefore, for a given small $h_0>0$, up to an affine transformation, 
\[ \{w<h_0\}\approx B_{\sqrt{h_0}}\quad\mbox{and}\quad \partial w(\{w<h_0\})\approx B_{\sqrt{h_0}}, \quad\mbox{ as }\delta \rightarrow 0.\]

\vskip 5pt

In the following lemma, we show that the sub-level sets of $u'$ and their images are close to ellipsoids with controlled eccentricity, 
and $u'$ is close to a parabola given by the second order Taylor expansion of $w$ in a small sub-level set $S_{h_0}[u']$ defined by \eqref{sections}. 

\begin{lemma}\label{l2}
For any given constant $\tilde{\eta}_0>0$ small, 
there exist small positive constants $h_0=h_0(\tilde{\eta}_0)$, $\delta_0=\delta_0(h_0,\tilde{\eta}_0)$,
and a symmetric matrix $A$ with $\|A\|$, $\|A^{-1}\|\leq K$ (a universal constant)
and $\det\ A=1$ such that
\begin{equation}\label{incl1}
A\left(B_{\sqrt{h_0/3}} \right)\cap \{x_n> {\tilde\rho}(x')\} \subset S_{h_0}[u']\subset
A\left(B_{\sqrt{3h_0}}\right)\cap \{x_n> {\tilde\rho}(x')\},
\end{equation}
\begin{equation}\label{incl2}
A^{-1}\left(B_{\sqrt{{h_0}/{3}}} \right)\cap \{x_n> {\tilde\rho}^*(x')\} \subset  \partial u'(S_{h_0}[u'])\subset
A^{-1}\left(B_{\sqrt{3h_0}} \right)\cap \{x_n> {\tilde\rho}^*(x')\},
\end{equation}
provided $\delta<\delta_0$.
Moreover $$\|u'-\frac{1}{2}|A^{-1}x|^2\|_{L^\infty\left(A\left(B_{\sqrt{h_0/3}}\right)\cap\{x_n> {\tilde\rho}(x') \}\right)}\leq \tilde{\eta}_0 h_0, $$
and $A^{-1}(\textbf{e}_n)$ is parallel to $A(\textbf{e}_n)$.
\end{lemma}

\begin{proof}
The proof is similar to that in \cite{CF1}, and the main steps are outlined as follows. 
Denote by $\omega_0$ the term $\omega(\delta_0)$ in \eqref{cpt11}.
Then, from \eqref{cpt11} and \eqref{Taylorw} we have 
\begin{equation}
\|u'-\frac{1}{2}D^2w(0)x\cdot x\|_{L^\infty(E_{4h_0}\cap\{x_n> {\tilde\rho}(x')\})} \leq \omega_0+|Dw(0)\cdot x|+O(h_0^{\frac{3}{2}})
\end{equation}
where $E_{h_0}:=\{x:\frac{1}{2}D^2w(0)x\cdot x\leq h_0\}.$
Therefore, by taking $h_0, \delta_0$ sufficiently small we can obtain 
\begin{equation}\label{d2wapp}
\|u'-\frac{1}{2}D^2w(0)x\cdot x\|_{L^\infty(E_{4h_0}\cap\{x_n> {\tilde\rho}(x')\})} \leq \frac{1}{2}\tilde{\eta}_0 h_0.
\end{equation}

Note that $\frac{1}{C}I\leq D^2w(0)\leq CI$ for some universal constant $C$.
By symmetry, $w_n=0$ on $\mathcal{P}:=\{x\in B_{1/5} : x_n=-\delta_1\}$, and thus $w_{ni}=0$ on $\mathcal{P}$ for all $1\leq i\leq n-1$. 
Since $w\in C^3(B_{1/5})$ and $\delta_1\to0$ as $\delta\to0$, we obtain at the origin $D^2w(0)=D^2w(z)+O(\delta)$, where $z=(0,\cdots,0,-\delta_1)$, and 
\[ w_{ni}(0) = O(\delta),\quad\mbox{ for all $i=1,\cdots,n-1$.} \] 
Hence, one can find a symmetric matrix $A$ satisfying $\|A-[D^2w(0)]^{-1/2}\| = O(\delta)$ such that $A^{-1}(e_n)$ is parallel to $A(e_n)$. 
This implies that $\|A^{-1}\|, \|A\|\leq K$, for a universal constant $K$.
And from \eqref{d2wapp} we obtain
\begin{equation}
\|u'-\frac{1}{2}|A^{-1}x|^2\|_{L^\infty(B_{2\sqrt{h_0}}\cap\{x_n> {\tilde\rho}(x')\})} \leq \tilde{\eta}_0 h_0
\end{equation}
which gives \eqref{incl1} and the second inclusion of \eqref{incl2}.
We refer the reader to \cite[Lemma 4.3]{CF1} for more detailed computation on the matrix $A$. 

To prove the first inclusion of \eqref{incl2}, we need to use the Legendre transform of $u'$,  namely let 
$u^*:B_{2\sqrt{h_0}}\cap\{x_n> {\tilde\rho^*}(x')\} \rightarrow \mathbb{R}$ be the convex function defined by
$$u^*(y):=\sup_{B_{2\sqrt{h_0}}\cap\{x_n> {\tilde\rho}(x')\}}\left\{ x\cdot y-u'(x) \right\}.$$ 
Then one can verify that 
$$\|u^*-\frac{1}{2}|Ay|^2\|_{L^\infty(B_{2\sqrt{h_0}}\cap\{x_n> {\tilde\rho}^*(x')\})}\leq 2\tilde{\eta}_0h_0.$$ 
By the standard property of Legendre transform 
$B\subset \partial u'(\partial u^*(B))$ for any Borel set $B $,  we can easily get the desired inclusion from the previous estimate.  
\end{proof}

\vskip5pt

\subsection{Iteration argument} \label{ss24}
Let $u_1=u'$, $U_1=U'$ and $U^*_1={U^*}'$. Then we have the initial setting in Lemma \ref{l2} for $u_1$, which is the potential function of the optimal transport from $(U_1, f_1)$ to $(U^*_1, g_1)$, where $f_1, g_1$ are the rescaled densities in Lemma \ref{rele}. 
Let $A_1=A$ be the symmetric matrix in Lemma \ref{rele}.
Now make the rescaling 
\beq\label{res2} 
x\mapsto \frac{1}{\sqrt{h_0}}A_1^{-1}x,
\eeq
and define
\begin{align*}
  u_2(x) &= \frac{1}{h_0}u_1(\sqrt{h_0}A_1x), \\
 f_2(x) &=f_1(\sqrt{h_0}A_1x),\\
  g_2(x) &=g_1(\sqrt{h_0}A_1^{-1}x).
\end{align*}
Moreover, let $U_2=S_{1}[u_2]$ and $U^*_2=\p u_2(S_{1}[u_2])$.
Thanks to Lemma \ref{l2}, we have $u_2$ is the potential function of the optimal transport from $(U_2, f_2)$ to $(U^*_2, g_2)$ satisfying all conditions of Lemma \ref{rele}.
Therefore, we can apply the argument in \S\ref{ss23} to $u_2$.  

Similarly, for $k=1,2,\cdots$, by the rescaling $x\mapsto \frac{1}{\sqrt{h_0}}A_k^{-1}x$, letting
\begin{align*}
 u_{k+1}(x) &=\frac{1}{h_0}u_k(\sqrt{h_0}A_kx), \\
 f_{k+1}(x) &=f_k(\sqrt{h_0}A_kx),\\ 
 g_{k+1}(x) &=g_k(\sqrt{h_0}A_k^{-1}x),
\end{align*}
and from Lemma \ref{l2} we can find a symmetric matrix $A_{k+1}$ satisfying
\begin{align*}
 \|A_{k+1}\|, \|A_{k+1}^{-1}\| \leq K,&\ \quad \det\, A_{k+1}=1, \\
 A_{k+1}\left(B_{\sqrt{h_0/3}} \right)\cap \{x_n> {\tilde\rho}(x')\} \subset &\ S_{h_0}[u_{k+1}]\subset A_{k+1}\left(B_{\sqrt{3h_0}}\right)\cap \{x_n> {\tilde\rho}(x')\}, \\
 A_{k+1}^{-1}\left(B_{\sqrt{{h_0}/{3}}} \right)\cap \{x_n> {\tilde\rho}^*(x')\} \subset &\  \partial u_{k+1}(S_{h_0}[u_{k+1}])\subset A_{k+1}^{-1}\left(B_{\sqrt{3h_0}} \right)\cap \{x_n> {\tilde\rho}^*(x')\}, \\
 \|u_{k+1}-\frac{1}{2}|A_{k+1}^{-1}x|^2\|&_{L^\infty\left(A_{k+1}\left(B_{\sqrt{h_0/3}}\right)\cap\{x_n> {\tilde\rho}(x') \}\right)}\leq \tilde{\eta}_0 h_0,
\end{align*}
where $K$ and $h_0$ are the same as in Lemma \ref{l2}, while $\tilde\rho, \tilde\rho^*$ are the rescaled boundary functions, which tend to be flat as $k\to\infty$.

Let 
\[ M_k:=A_k\cdot\ldots\cdot A_1,\quad\mbox{ for all }k=1,2,\cdots. \]
We obtain a sequence of symmetric matrices satisfying 
\[ \|M_k\|, \|M_k^{-1}\|\leq K^k,\quad\mbox{and}\quad \det\, M_k =1,\quad\forall k\geq1.  \]
From the above iteration we have
\begin{align*}\label{e55}
 M_k\left(B_{(h_0/3)^{k/2}}\right)\cap\{x_n> {\tilde\rho}(x')\}
      \subset &\ S_{h_0^k}[u_1]
      \subset M_k\left(B_{(3h_0)^{k/2}}\right)\cap\{x_n> {\tilde\rho}(x')\}, \\
 M_k^{-1}\left(B_{(h_0/3)^{k/2}}\right)\cap\{x_n> {\tilde\rho}^*(x')\}
        \subset &\ \partial u_1\left(S_{h_0^k}[u_1]\right)
        \subset M_k^{-1}\left(B_{(3h_0)^{k/2}}\right)\cap\{x_n> {\tilde\rho}^*(x')\}. \nonumber
\end{align*}
Hence,
\begin{equation}\label{ho1}
B_{\left(\frac{\sqrt{h_0}}{\sqrt{3}K}\right)^{k}}\cap\{x_n> {\tilde\rho}(x')\}
        \subset S_{h_0^k}[u_1]
        \subset B_{(\sqrt{3}K\sqrt{h_0})^{k}}\cap\{x_n> {\tilde\rho}(x')\}\quad\forall k\geq1. 
\end{equation}

For any given $\beta\in(0,1)$, by choosing $h_0$ and $\delta_0$ small enough we can show that $u_1$ is $C^{1,\beta}$ at the origin, and thus obtain Lemma \ref{alpha1}.
\begin{proof}[Proof of Lemma \ref{alpha1}]
Fix $\beta\in(0,1)$, and let $r_0:={\sqrt{h_0}}/({\sqrt{3}K})$.
From \eqref{ho1} we have
	\[ \|u_1\|_{L^\infty(B_{r_0^k}\cap\{x_n>\tilde\rho(x')\})} \leq h_0^k = (\sqrt{3}Kr_0)^{2k} \leq r_0^{(1+\beta)k}, \]
provided $h_0$ (and so $r_0$) is sufficiently small. 
This implies the $C^{1,\beta}$ regularity of $u_1$ at the origin. 
By rescaling back to the original solution and the arbitrariness of the boundary point $x_0\in\p\Lambda$, we obtain $u\in C^{1,\beta}(\overline\Lambda)$ and finish the proof.
\end{proof}

\vskip10pt

\section{Proof of Theorems \ref{ptheorem} and \ref{2ptheorem}}

\subsection{$C^{2,\alpha}$ estimate}

We can adapt a perturbation argument from \cite[\S5]{CLW} to prove Theorem \ref{ptheorem}.
By changing coordinates and subtracting a linear function, we assume $0\in\p\Lambda$, $u\geq0$, $u(0)=0$ and $Du(0)=0$. 
From Lemma \ref{alpha1}, we see that for any fixed $\epsilon>0$ small, 
$$B_{C^{-1}h^{\frac{1}{2}+\epsilon}}\cap \{x_n> \rhoprime (x')\}
            \subset S_h[u]\subset B_{Ch^{\frac{1}{2}-\epsilon}}\cap \{x_n> \rhoprime (x')\},$$
provided $\delta_0$ is sufficiently small.
Now we construct an approximate solution of $u$ in $S_h[u]$ as follows. 
Denote 
$$D_h^+=S_h[u]\cap \{x_n\geq h^{1-3\epsilon}\} . $$  
When $h>0$ is sufficiently small, we have $D_h^+\Subset  \Lambda.$
Let $u^*$ be the dual potential function, that is the Legendre transform of $u$. 
The proof of Lemma \ref{alpha1} applies also to $u^*$, namely $u^*\in C^{1,\beta}(\overline{\Lambda^*})$ for any given $\beta\in (0,1)$. 
Hence, for any 
$x\in D_h^+$, we have $u_n(x)\geq 0$. Otherwise, one has $\dist(Du(x), \partial \Lambda^*)\lesssim h^{1-2\epsilon}$,
but from the $C^{1,\beta}$ estimate of $u^*$,  
\begin{equation*}
\begin{split}
	\dist(x, \partial\Lambda) &=\dist(Du^*(Du(x)), \partial \Lambda) \\
	&\lesssim h^{(1-2\epsilon)(1-\epsilon)} \ll h^{1-3\epsilon} 
\end{split}
\end{equation*}
provided $h$ is sufficiently small, which contradicts to the definition of $D_h^+ $.  

Let $D_h^-$ be the reflection of $D_h^+$ with respect to the hyperplane $\{x_n=h^{1-3\epsilon}\} $.  
Denote $D_h=D_h^+\cup D_h^- $.  
By the property that $u_n\big|_{D_h^+}\geq 0$, 
we see that $D_h$ is a convex set.
Now, let $w$ be the solution of 
\begin{equation}
\label{eq:MA v}
\begin{cases}
\det(D^2w)=1&\mbox{in $D_h$},\\
w=h&\mbox{on $\partial D_h$}.
 \end{cases}
 \end{equation}
Our proof relies on the following lemma.
\begin{lemma}\label{diff1}
Assume that
$$\left|\frac{f(x)}{g(Du(x))}-1\right|\lesssim h^\tau \ \ \mbox{in}\ \ D_h.$$  
Then we have  
$$\|u-w\|_{L^\infty(D_h\cap \Lambda)}\lesssim h^{1+\tau}.$$
\end{lemma}
	  
\begin{proof}
The proof uses a similar idea as that of Theorem \ref{main} (i) in \cite[\S5]{CLW}.
Divide $\partial D_h^+ = \mathcal{C}_1\cup\mathcal{C}_2$ into two parts, 
where $\mathcal{C}_1\subset\{x_n> h^{1-3\epsilon}\}$ and $\mathcal{C}_2\subset\{x_n= h^{1-3\epsilon}\}$.  
On $\mathcal{C}_1$ we have $u=w$.  
On $\mathcal{C}_2$, by symmetry we have $D_nw=0$.  
We \emph{claim} that $0\leq D_nu \leq C_1h^{1-4\epsilon}$ on $\mathcal{C}_2$, 
for any given small $\epsilon>0$.  

To see this, for any $x=(x',x_n)\in \mathcal{C}_2 $,  let $z=(x', \rhoprime (x'))$ be the point on $\partial\Lambda $.  
Since $Du(\partial \Lambda)\subset \partial \Lambda^*$ 
and $u\in C^{1,1-\epsilon}(\bar{\Lambda})$, for any $\epsilon\in(0,1)$ (by Lemma \ref{alpha1}), 
it is straightforward to compute that $|D_nu(z)|\leq Ch^{2(\frac{1}{2}-\epsilon)(1-\epsilon)} $.  On the
other hand $|D_nu(x)-D_nu(z)|\leq Ch^{(1-3\epsilon)(1-\epsilon)} $.  Hence $0\leq D_nu(x) \leq C_1h^{1-4\epsilon}$, provided $\epsilon$ is sufficiently small.

Let 
$$\hat w = (1-h^\tau)^{1/n}w-(1-h^\tau)^{1/n}h+h, $$
and 
$$\check w = (1+h^\tau)^{1/n}w-(1+h^\tau)^{1/n}h+h+C_1(x_n-Ch^{1/2-\epsilon})h^{1-4\epsilon}. $$ 
Then
	\begin{align*}
		\det\,D^2\hat w \leq \det\,D^2 u \leq \det\,D^2\check w  & \quad\mbox{ in }  D_h^+,\\
			\check w\leq u=\hat w  = h &\quad\mbox{ on }\mathcal{C}_1, \\
		D_n\hat w =0 < D_n u <D_n\check w&\quad\mbox{ on }\mathcal{C}_2.
	\end{align*}
By comparison principle, we have $\hat w \geq u\geq \check w$ in $D_h^+$.  

Since $h>0$ is small, $\tau<1/2$, and $\epsilon>0$ is  small, we obtain
	\begin{equation}\label{coreL}
		|u-w| \leq Ch^{1+\tau} \qquad\mbox{in } D_h^+. 
	\end{equation}

\vspace{5pt}

Next, we estimate $|u-w|$ in $ D_h\cap \Lambda$.  
For $x=(x',x_n)\in D_h^-\cap \Lambda$, let $z=(x', 2h^{1-3\epsilon}-x_n) \in D_h^+$.
Then  $|x-z|\leq Ch^{1-3\epsilon}$.  
From \eqref{coreL}, $|u(z)-w(z)| \leq Ch^{1+\tau}$.
Since $w$ is symmetric with respect to $\{x_n=h^{1-3\epsilon}\}$, we have  $w(x)=w(z)$.
Since $u\in C^{1,1-\epsilon}(\bar{\Lambda})$, we obtain
	\begin{equation*}
		|u(x)-u(z)| \leq \|Du\|_{L^\infty(D_h)} |x-z| \leq Ch^{3/2-3\epsilon}.
	\end{equation*}
Therefore, for the given constant $\tau\in(0,\frac12)$,  
	\begin{equation*}
		|u(x)-w(x)| \leq |u(x)-u(z)| + |u(z)-w(z)| \leq Ch^{1+\tau}.
	\end{equation*}
Combining with \eqref{coreL} we thus obtain the desired $L^\infty$ estimate
	\begin{equation}\label{coreLL}
		|u-w| \leq Ch^{1+\tau}\qquad\mbox{in }  D_h\cap \Lambda
	\end{equation}
\end{proof}

Having Lemma \ref{diff1} in hand,
we can prove Theorem \ref{ptheorem} by following  the proof of Theorem \ref{main} (i) as in \cite[\S5]{CLW}.
Here, we outline the main steps as follows. 
\begin{proof}[Proof of Theorem \ref{ptheorem}]
Let $D_k=D_{h_k},$ where $h_k= 4^{-k}$, $k=0, 1,2,\cdots.$  
To obtain the $C^{1,1}$ estimate of $u$ is equivalent to show that $D_k$ has good shape for all $k$, namely the ratio of the largest radius and the smallest radius of its minimal ellipsoid is uniformly bounded. 
This is done by an induction argument.

Let $u_k$, $k=0,1,2, \cdots$, be the convex solution of
	\begin{align}
		\det\,D^2 u_k = 1 &\qquad\mbox{in }D_k,\\
		u_k = h_k & \qquad\mbox{on }\partial D_k. \nonumber
	\end{align}
Suppose $D_k$ has good shape for all $k\leq N.$
Then by Lemma \ref{diff1} and Schauder estimate (see \cite[Lemma 5.4]{CLW}), we have
	\begin{equation}\label{in1}
		|D^2u_i(x)-D^2u_{i+1}(x)| \leq Ch_{i}^\tau,
	\end{equation}
	for $x\in D_{i+2}$ and $0\leq i\leq N$, where the constant $\tau=\frac{\alpha}{2}\in(0,\frac12)$. 
Therefore, 
\begin{equation}\label{induc}
|D^2u_{N+1}(0)|\leq |D^2u_0(0)|+ \sum_{i=0}^N|D^2u_i(0)-D^2u_{i+1}(0)|\leq C+ \sum_{i=0}^N Ch_i^\tau\leq C_1
\end{equation}
 for some universal constant $C_1.$
Then, by \cite[Lemma 5.3]{CLW} we have that $D_{N+1}$ also has good shape. Hence by induction on $k$ we see that $D_k$ has good shape for all $k.$

To obtain the $C^{2,\alpha}$ estimate, for any given point $z\in\overline\Lambda$ near the origin such that $4^{-k-4}\leq u(z)\leq 4^{-k-3}$
we only need to estimate $$|D^2u(z)-D^2u(0)|\leq |D^2u(z)-D^2u_k(z)|+|D^2u_k(z)-D^2u_k(0)|+|D^2u_k(0)-D^2u(0)|.$$
Since $f\in C^\alpha(\overline\Lambda)$, $g\in C^\alpha(\overline{\Lambda^*})$, similarly as in \cite[\S5]{CLW} we can obtain
\[ |D^2u(z)-D^2u(0)| \leq C|z|^\alpha, \]
which gives the H\"older continuity of $D^2u$ at the boundary. 
Combining with the interior $C^{2,\alpha}$ estimates in \cite{C1,JW}, we have the global $C^{2,\alpha}$ regularity in Theorem \ref{ptheorem}. 
\end{proof}

\vskip5pt

\subsection{$W^{2,p}$ estimate}

When the densities $f, g$ satisfy  \eqref{BC}, \eqref{fg} and are continuous, 
we can also obtain the global $W^{2,p}$ estimate in Theorem \ref{2ptheorem}.
As seen in \S\ref{s2}, for the H\"older continuous densities, the solution $u$ can be approximated by $\tilde u\in C^{2,\alpha}$ that is a potential function over convex domains, and thus one has \eqref{cd0} at the initial step. 
For continuous densities, we do not have the initial estimate \eqref{cd0} since such an approximate solution may not even be $C^{1,1}$. 

To overcome this difficulty we use a similar technique as in \cite{Chen} by directly exploiting the sub-level sets. 
Eventually we can also establish Lemma \ref{alpha1} for continuous densities. 
Once having $u\in C^{1,\beta}(\overline\Lambda)$ for all $\beta\in(0,1)$, the global $W^{2,p}$ estimate follows from the argument as in \cite{Sa1}, (see also the proof of Theorem \ref{main} (ii) in \cite{CLW}).

\begin{proof}[Proof of Theorem \ref{2ptheorem}]
In order to obtain Lemma \ref{alpha1} for continuous densities, we follow the four steps as in \S\ref{s2}. 
First, let $\tilde u$ be the potential function of optimal transport from $(\Omega, \tilde f)$ to $(\Omega^*, \tilde g)$, where $\Om, \Om^*$ are convex domains and $\tilde f, \tilde g$ are the extended continuous densities. 
Let $0\in \p\Lambda$, without loss of generality we may assume that $0=u(0)=\tilde u(0)$ and $D\tilde u(0)=0$.
By a compactness argument as in Lemma \ref{vclose}, we have
	\[ \|u-\tilde u\|_{L^\infty(\Lambda)} \leq \omega(\delta) \]
for a nondecreasing function $\omega:\R_+\to\R_+$ with $\omega(\delta)\to0$ as $\delta\to0$. 

Next we localise the problem by normalising a small sub-level set $S_h[\tilde u]$ of $\tilde u$ at the origin, where $h>0$ is a fixed small constant. 
From Theorem \ref{main} (ii), there is a universal constant $C$ and a unimodular matrix $A$ such that 
\begin{align*}
 A(B_{\sqrt{h}/C})\cap\Om \subset &\ S_h[\tilde u] \subset A(B_{C\sqrt{h}})\cap\Om,\\
 \|A\|, \|A^{-1}\| &\ \leq Ch^{-\epsilon},
\end{align*}
where $\epsilon>0$ can be as small as we want. 
Make the rescaling $x\mapsto \frac{1}{\sqrt{h}}A^{-1}x$ and define
\begin{align*}
 u_1(x)&=\frac{1}{h}u(\sqrt{h}Ax),\\
  \tilde u_1(x)&=\frac{1}{h}\tilde u(\sqrt{h}Ax), 
\end{align*}
and accordingly
\begin{align*}
 f_1(x)&=f(\sqrt{h}Ax),\\ 
 g_1(x)&=g(\sqrt{h}(A^t)^{-1}x),
\end{align*}
where $A^t$ is the transpose of $A$. 
Correspondingly, the domains become $\Om_1=\frac{1}{\sqrt{h}}A^{-1}\Om$ and $\Om^*_1=\frac{1}{\sqrt{h}}A^t\Om^*$.
From the proof of Lemma \ref{rele}, one has the scaled densities $f_1, g_1$ tend to the constant and the domains $\Om_1, \Om^*_1$ tend to be flat near the origin, as $h\to0$. 
Moreover, for any given $\eta_0>0$ small, one has
\beqs
\|u_1-\tilde u_1\|_{L^\infty(S_1[\tilde u_1])} \leq \omega\left(\frac{\delta}{h}\right) < \eta_0,
\eeqs
provided $\delta$ is sufficiently small. 

Then we construct an optimal transport problem with constant densities. 
Similarly as in \S\ref{ss23}, define the domains $\mathcal{D}_1$ and $\mathcal{D}_2$.
Let $w$ be the convex function satisfying $w(0)=u_1(0)=0$ and
\beq\label{mae02}
\bigg\{{\begin{aligned}
	 &\det\, D^2w(x) = 1 \ \ \text{in}\ \mathcal D_1,\\
	 &Dw(\mathcal D_1) =\mathcal D_2.
	 \end{aligned}} \eeq
Analogous to \eqref{cpt11} and \eqref{ap2}, we can then obtain
\beqs
\|w-u_1\|_{L^\infty(B_{{1}/{3}}\cap \{x_n>{\tilde\rho}(x')\})} 
    \leq \omega(\delta)\rightarrow 0,\quad \ \text{as}\ \  \delta\rightarrow 0.
\eeqs
and thus
\beqs
 \|w-\tilde u_1\|_{L^\infty(B_{{1}/{4}}\cap \{x_n>{\tilde\rho}(x')\}  )}
      \leq \omega(\delta)+\eta_0\rightarrow 0,\quad \ \text{as}\ \  \delta, \eta_0\rightarrow 0.
\eeqs
From \cite[Lemma 3.4]{Chen} and $\tilde u_1\in C^{1,\beta}$ for all $\beta\in(0,1)$, one has for any $x\in B_{{1}/{5}}\cap \{x_n>{\tilde\rho}(x')\}$ and for any $p\in\p w(x)$,
\beqs
\big|p-D\tilde u_1(x)\big| \leq C\big( \omega(\delta)+\eta_0 \big)^{\beta/2}.
\eeqs
Hence, by the symmetry of $w$, we have $\partial w(B_{1/5})\subset B_{1/4}$, provided $\delta, \eta_0$ are sufficiently small.
Since $w$ is a solution to \eqref{mae02}, by the interior estimates \cite{GT} we obtain 
\beq\label{intC3}
\|w\|_{C^3(B_{1/5})}\leq C
\eeq
for a universal constant $C>0$. 

Once having the smooth approximate solution $w$, we can similarly obtain Lemma \ref{l2}, namely $u_1$ is close to a parabola given by the second order Taylor expansion of $w$ in a small sub-level set $S_{h_0}[u_1]$. 
By the iteration argument, we then have \eqref{ho1}, which implies Lemma \ref{alpha1}, namely for any given $\beta\in(0,1)$, there is a small constant $\delta_0>0$ such that the original solution $u\in C^{1,\beta}(\overline\Lambda)$ provided $\delta<\delta_0$. 
Finally, the global $W^{2,p}$ estimate can be obtained by using a covering argument from \cite{Sa1}, see \cite{CF, CLW} for more details.
\end{proof}

\vskip10pt

\section{Some applications}

In the last section we give some interesting applications of our global regularity of optimal mappings in non-convex domains.  

\subsection{Free boundary problem}
As in \cite{Chen1} we discuss a model of free boundary problem arising in optimal transportation.

Let $\Lambda$ and $\Lambda^*$ be two bounded domains in $\R^n$, associated with densities $f$ and $g$, respectively. 
Let $m$ be a positive number satisfying
\beq\label{pmass}
	m\leq \min\left\{ \int_\Lambda f,\ \int_{\Lambda^*}g \right\}.
\eeq
Let the cost be the quadratic cost. 
The optimal partial transport problem asks for the optimal mapping that minimising the cost transporting mass $m$ from $\Lambda$ to $\Lambda^*$. 
The portion $U\subset\Lambda$ been transported is called the \emph{active region}.
In \cite{CM}, Caffarelli and McCann proved that the free boundary $\p U\cap\Lambda$ is $C^{1,\alpha}$. 
Assuming $\Lambda, \Lambda^*$ are $C^2$, uniformly convex, and the distance $\dist(\Lambda, \Lambda^*)$ is sufficiently large, the first author \cite{Chen1} obtained the $C^{2,\alpha}$ regularity of the free boundary $\p U\cap\Lambda$. 

The key observation is that when $\dist(\Lambda, \Lambda^*)$ is sufficiently large, for any $x\in\Lambda, y\in\Lambda^*$, $\frac{y-x}{|y-x|}$ is uniformly close to some unit vector $e$. It is known that for $x\in\p U\cap\Lambda$, the unit inner normal of the free boundary $\p U$ is given by
\beqs
\nu(x)=\frac{Du(x)-x}{|Du(x)-x|}\to e,\quad\mbox{ as }\quad \dist(\Lambda, \Lambda^*)\to\infty.
\eeqs
Therefore, when $\dist(\Lambda, \Lambda^*)$ is sufficiently large, the active region $U$ is a small perturbation of a convex domain. 
By applying our Theorems \ref{ptheorem} and \ref{2ptheorem}, we can obtain the following:

\begin{corollary}
Let $\Lambda$  and $\Lambda^*$ be $C^{1,1}$ domains that are
$\delta$-close to $\Omega$  and $\Omega^*$ in $C^{1,1}$ norm, respectively,
where $\Om$ and $\Om^*$ are bounded convex domains with $C^{1,1}$ boundaries. 
Let $m$ satisfying \eqref{pmass} be the mass to transport, and $U$ be the active region.
Then:
\begin{itemize}
\item[(i)] when $f, g$ are continuous, for any given $\beta\in(0,1)$, there exists a small constant $\delta_0>0$ and a large constant $L$ such that $\p U\cap\Lambda$ is $C^{1,\beta}$, provided $\delta<\delta_0$ and $\dist(\Lambda, \Lambda^*)\geq L$;

\item[(ii)] when $f, g$ are $C^\alpha$ for some $\alpha\in(0,1)$, there exists a small constant $\delta_0>0$ and a large constant $L$ such that $\p U\cap\Lambda$ is $C^{2,\alpha}$, provided $\delta<\delta_0$ and $\dist(\Lambda, \Lambda^*)\geq L$.
\end{itemize}
\end{corollary}
Note that the above regularity is interior regularity, namely for any $\Lambda'\Subset\Lambda$, the $C^{1,\beta}$ (or $C^{2,\alpha}$) norm of $\p U\cap\Lambda'$ depends also on the domain $\Lambda'$.

\subsection{A singularity model}

Consider an optimal transport problem from a source domain $\Lambda$ with density $f$ to the target $\Lambda^*=\Lambda^*_1\cup\Lambda^*_2$ with density $g$, where $\Lambda^*_1$ and $\Lambda^*_2$ are two domains separated by a hyperplane $H$, and the densities satisfy $C^{-1}\leq f, g \leq C$ and
\beqs
\int_\Lambda f = \int_{\Lambda^*}g. 
\eeqs
Let the cost be the quadratic cost, $u$ be the convex potential of the optimal transport from $(\Lambda, f)$ to $(\Lambda^*,g)$.
Then its Legendre transform $u^*$ is the convex potential of the optimal transport from $(\Lambda^*,g)$ to $(\Lambda, f)$.
In \cite{Chen1}, it was proved that the domains $U_1:=\p u^*(\Lambda^*_1)$ and $U_2:=\p u^*(\Lambda^*_2)$ are separated by the free boundary $\mathcal{F}\subset \Lambda$, and when $\dist(\Lambda^*_1, \Lambda^*_2)$ is sufficiently large, the free boundary $\mathcal{F}$ is close to a hyperplane. See also some related discussion in \cite{KM}.
Hence, by applying our Theorems \ref{ptheorem} and \ref{2ptheorem}, we have:
\begin{corollary}
Let $\Lambda, \Lambda^*_1, \Lambda^*_2$ be $C^{1,1}$ domains that are $\delta$-close to $\Omega, \Omega^*_1, \Omega^*_2$ in $C^{1,1}$ norm, respectively,
where $\Om, \Om^*_1, \Om^*_2$ are bounded convex domains with $C^{1,1}$ boundaries. 
Let $\Lambda'\Subset\Lambda$, then:
\begin{itemize}
\item[(i)] when $f, g$ are continuous, for any given $\beta\in(0,1)$, there exists a small constant $\delta_0>0$ and a large constant $L$ such that $\mathcal{F} \cap\Lambda'$ is $C^{1,\beta}$, provided $\delta<\delta_0$ and $\dist(\Lambda, \Lambda^*)\geq L$;

\item[(ii)] when $f, g$ are $C^\alpha$ for some $\alpha\in(0,1)$, there exists a small constant $\delta_0>0$ and a large constant $L$ such that $\mathcal{F}\cap\Lambda'$ is $C^{2,\alpha}$, provided $\delta<\delta_0$ and $\dist(\Lambda, \Lambda^*)\geq L$.
\end{itemize}
\end{corollary}

\subsection{Minimal Lagrangian diffeomorphisms}

In \cite{Wo}, Wolfson studied minimal Lagrangian diffeomorphisms between simply connected domains in $\R^2$. 
The problem is as follows: given $D_1$ and $D_2$ two simply connected domains in $\R^2$ with smooth boundaries and with equal areas, find a diffeomorphism $\psi : D_1\to D_2$ smooth up to the boundaries such that the graph of $\psi$ is a minimal Lagrangian surface in $\R^4$. 
Such a mapping $\psi$ is called a \emph{minimal Lagrangian diffeomorphism} from $D_1$ to $D_2$. 

An equivalent statement is that there is a solution $u\in C^\infty(\overline D_1)$ of the second boundary problem for the Monge-Amp\`ere equation
\beqs 
\bigg\{{\begin{aligned}
	 &\det\, D^2u = 1 \ \ \text{in}\ D_1,\\
	 &Du(D_1) = D_2.
	 \end{aligned}}
\eeqs
The equivalency can be seen that by choosing a suitable Lagrangian angle, the diffeomorphism $\psi = Du$ mapping from $D_1$ to $D_2$. 

Under the assumption that both $\p D_1$ and $\p D_2$ have positive curvatures, the existence of global smooth solutions was proved by Delano\"e \cite{D91}. The higher dimensional analogue of Delano\"e's result was proved by Caffarelli \cite{C96} and Urbas \cite{U1}. By our recent result Theorem \ref{main} in \cite{CLW}, we show that both $\p D_1$ and $\p D_2$ have non-negative curvatures guarantees the existence of minimal Lagrangian diffeomorphisms $\psi$. 
By applying Theorem \ref{ptheorem} we are able to further relax the assumption to the following:
\begin{corollary}
Assume that $D_1, D_2$ are $\varepsilon$-close to smooth convex domains $\Om_1, \Om_2$, respectively.  
Then there exists a small constant $\varepsilon_0>0$ such that 
there exists a minimal Lagrangian diffeomorphim $\psi : D_1\to D_2$, provided $\varepsilon<\varepsilon_0$.
\end{corollary}

\subsection{Optimal transportation with general costs}

The regularity of an optimal transport map with general costs has been studied by many researchers. 
In \cite{MTW}, Ma, Trudinger, and Wang found a fourth order condition, the so-called \emph{MTW condition}, of the cost function, which ensures the smoothness of the map. 
When the cost does not satisfy the MTW condition, but is a small perturbation of the quadratic cost, various regularity results have been obtained in \cite{Chen, CF, CF1}, see also \cite{CGN, DF}. 

We remark that the proofs of our Theorems \ref{ptheorem} and \ref{2ptheorem} also allow a small perturbation of the cost function. 
To be specific, assume that the cost function $c=c(x,y)$ satisfies
\begin{itemize}
\item[({\bf C0})] The cost function is of class $C^3$ with $\|c\|_{C^3(\Lambda\times\Lambda^*)}<\infty$.
\item[({\bf C1})] $\forall x\in\Lambda$, the map $\Lambda^*\ni y\mapsto D_xc(x,y)\in\R^n$ is injective. 
\item[({\bf C2})] $\forall y\in\Lambda^*$, the map $\Lambda\ni x\mapsto D_yc(x,y)\in\R^n$ is injective. 
\item[({\bf C3})] $\det(D_{xy}c)(x,y)\neq0$ for all $(x,y)\in \Lambda\times\Lambda^*$,
\end{itemize}
and 
\beq\label{pcost}
\|c - x\cdot y\|_{C^2(\Lambda\times\Lambda^*)} \leq \delta_1.
\eeq
The conclusion of Theorems \ref{ptheorem} and \ref{2ptheorem} remains true provided $\delta+\delta_1<\delta_0$ is sufficiently small, where $\delta$ is the perturbation of domains $\Lambda, \Lambda^*$ from convex domains, and $\delta_1$ is the perturbation of cost from the quadratic cost.

\end{document}